\documentclass[12pt]{amsart}

\usepackage{amssymb}

\usepackage{enumitem}

\usepackage{graphicx}

\makeatletter
\@namedef{subjclassname@2020}{%
	\textup{2020} Mathematics Subject Classification}
\makeatother

\usepackage[T1]{fontenc}
\newtheorem{theorem}{Theorem}[section]
\newtheorem{corollary}[theorem]{Corollary}
\newtheorem{lemma}[theorem]{Lemma}
\newtheorem{proposition}[theorem]{Proposition}

\theoremstyle{definition}

\newtheorem{remark}[theorem]{Remark}
\newtheorem{example}[theorem]{Example}

\numberwithin{equation}{section}

\begin{document}
	
	%%%%% To ease editing, for IMPAN journals add:
	
	\baselineskip=17pt
	
	%%%%%%%%%%%%%%%%
	
	\title[On a cyclic inequality and Shapiro-type analogues]{On a cyclic inequality with exponents and permutations, and its Shapiro-type analogues}
	
	\author[A. Czarnecki]{Andrzej Czarnecki}
	\address{Jagiellonian University, Faculty of Mathematics and Computer Science\\
		\L{}ojasiewicza 6\\
		30-848 Krakow, Poland}
	\email{andrzejczarnecki01@gmail.com}
	
	\author[G. Kiciński]{Gabriel Kici\'nski}
	\address{University of Warsaw, Faculty of Mathematics, Informatics and Mechanics\\
		Banacha 2\\
		02-097 Warszawa, Poland}
	\email{gabriel.kicinski@students.mimuw.edu.pl}
	
	\date{\today}

	\begin{abstract}
		We prove that the cyclic inequality $\sum\limits_{i=1}^{n}\left(\frac{x_i}{x_{i+1}}\right)^k\geq\sum\limits_{i=1}^{n}\frac{x_i}{x_{\sigma(i)}}$ holds for all positive $x_i$'s if and only if $k$ is in a specific range dependent on the permutation $\sigma$, related to band permutations. We also show that the same is not true for the Shapiro-type generalizations, proving in the process some analogous inequalities with exponents.
	\end{abstract}
	
	
	
	\maketitle
	
	\section{Introduction}
	Throughout this paper $n> 1$ will be a fixed number of real positive variables $x_{1},\ldots,x_{n}$. Every shift of indices by $p$ is understood to be cyclic, $x_{i+p}:=x_{i+p\ (\mathrm{mod }\ n)}$.%AC06.21 restated the global assumptions per the referee's request;
	
	This paper is a far-reaching offshoot of the International Tournament for Young Mathematicians 2019. The participants were asked to investigate, for a fixed natural $k$, a cyclic inequality
	\begin{align}\label{original}
		\sum\limits_{i=1}^{n}\left(\frac{x_i}{x_{i+1}}\right)^k\geq \sum\limits_{i=1}^{n}\frac{x_i}{x_{i+p-1}}    
	\end{align}
	amounting to the assertion that a cyclic shift by 1 is in a sense optimal among all cyclic shifts.
	
	The second author was able to give a broad generalization regarding any permutations of the indices, finding that for a fixed permutation $\sigma$, the values of $k$ for which the inequality
	\begin{align}\label{main}
		\sum\limits_{i=1}^{n}\left(\frac{x_i}{x_{i+1}}\right)^k\geq \sum\limits_{i=1}^{n}\frac{x_i}{x_{\sigma(i)}}
	\end{align}
	holds are strictly dependent on how much $\sigma$ displaces elements of $\{1,\ldots,n\}$. This is described in Theorem \ref{result} and Corollaries \ref{wieszift}, and \ref{mnieszift}. We point out that these corollaries on cyclic shifts are much more straightforward than the general case. With counterexamples provided after each theorem, the classification of ``exponent weights'' $k$ for which \eqref{main} holds is complete.
	
	We also give some remarks on the interpretation of the inequality, as well as on the class of permutations in question.
	
	The same dependency between exponent weights and permutation $\sigma$ does not occur for Shapiro-type inequalities
	\begin{align}\label{szapiro}
		\sum\limits_{i=1}^{n}\left(\frac{x_i}{x_{i+1}+x_{i+2}}\right)^k\geq \sum\limits_{i=1}^{n}\frac{x_i}{x_{\sigma(i)}+x_{\sigma^2(i)}},
	\end{align}
	which we will discuss in Section \ref{szap}. While it is well known that the behaviour of the original Shapiro inequality (inequality \eqref{szapiro} with $k=1$ and $\sigma=id$) depends on the number of variables (i.e. it holds only for even $n\leq 12$ and odd $n\leq 23$, cf. \cite{fink}) and is thus substantially different from the inequality \eqref{main}, it is still surprising that the inequality \eqref{szapiro} does not seem to depend on any relationship between $\sigma$ and $k$ (cf. Examples \ref{jeden}, \ref{nes}, and \ref{trzy} in the last section). We will distinguish between Shapiro-type inequalities with exponent weights (with arbitrary $k$ and $\sigma$), and Shapiro inequalities with exponent weights (with arbitrary $k$ but with such a $\sigma$ that the RHS is constant). We prove that \eqref{szapiro} never holds for the former, and that it does hold for the latter, showing counterexamples for generalizations of the so-called Nesbitt's inequalities in the process. In the present paper these inequalities serve ultimately only to exhibit lack of dependence between the permutation in question and the exponent.
	
	As mentioned before, every $x_i$ will be a positive real number. We will use small Greek letters for permutations of finite sets $\{1,\ldots,n\}$ or $\{1,\ldots,m\}$, in which indices $i$ and $j$ will lie. We reserve the letter $k$ for the exponent appearing in the inequalities.
	
	\section{Rearrangement inequality and inequality \eqref{main}.}\label{shift}
	
	\begin{theorem}[Rearrangement inequality, \cite{rearr} Theorem 368.]\label{rear}
		Given $m$ increasing sequences of $n$ non-negative numbers, $a_{(j,1)}\leq a_{(j,2)}\leq\ldots\leq a_{(j,n)}$, $ j\in\{1,\ldots,m\}$, we have $$\sum\limits_{i=1}^{n} \prod\limits_{j=1}^{m} a_{(j,i)} \geq \sum\limits_{i=1}^{n} \prod\limits_{j=1}^{m} a_{(j,\sigma_j(i))}$$ for any collection $\{\sigma_j\}_{1}^{m}$ of permutations of $\{1,\ldots,n\}$.  
	\end{theorem}
	Define $a_i:=\frac{x_i}{x_{i+1}}$ (observe that $\prod\limits_{i=1}^{n}a_i=1$), $b_i:=a_i^{-1}$. We will keep this notation throughout the paper.
	
	Our main result is the following theorem.
	\begin{theorem}\label{result}
		For a fixed $\sigma$, the inequality \eqref{main} holds if either of the following conditions is satisfied:
		\begin{enumerate}
			\item $k\geq 0$, and $\begin{cases}
			\ \ k\geq \sigma(i)-i &\text{ for every } i \leq \sigma (i) \\
			\ \ k\geq n+\sigma(i)-i &\text{ for every } i> \sigma(i)
			\end{cases}$
			\item $k\leq 0$, and $\begin{cases}
			-k\geq n+i-\sigma(i) &\text{ for every } i < \sigma (i) \\
			-k\geq i-\sigma(i) &\text{ for every } i\geq \sigma(i).
			\end{cases}$
		\end{enumerate}
		Moreover, the examples below show that either of the conditions is necessary.
	\end{theorem}
	This theorem has two immediate corollaries on cyclic shifts.
	\begin{corollary}\label{wieszift}
		If $k \in \mathbb{R}$, $k \geq 0$, the inequality \eqref{original} holds for $k \geq p-1$.
	\end{corollary}
	\begin{corollary}\label{mnieszift}
		If $k \in \mathbb{R}$, $k<0$, the inequality \eqref{original} holds for $-k \geq n- p+1$.
	\end{corollary}
	\begin{proof}[Proof of Theorem \ref{result}]
		Let $k=\frac{u}{v}$ be a rational non-negative number satisfying the assumptions of the theorem. Define $\Delta_n(i):=\begin{cases}0 & \text{ for } i \leq \sigma(i)\\
		n & \text{ for } i > \sigma(i)
		\end{cases}$. We now rewrite the inequality \eqref{main} in the following way:
		\begin{align*}
			&\sum\limits_{i=1}^{n}\left(\frac{x_i}{x_{i+1}}\right)^k \geq \sum\limits_{i=1}^{n} \frac{x_i}{x_{\sigma (i)}} \iff \\
			&\sum\limits_{i=1}^{n}a_i^k \geq \sum\limits_{i=1}^{n} \prod _{j=i}^{\sigma (i)-1 +\Delta_n(i)} a_j \iff \\
			&\sum\limits_{i=1}^{n}a_i^k \geq \sum\limits_{i=1}^{n} \left( \prod \limits_{j=i}^{\sigma (i)-1 +\Delta_n(i)} a_j \cdot \left( \prod\limits_{j=1}^n a_j\right)^{\frac{k+i-\sigma(i) -\Delta_n(i)}{n}}\right) \iff \\
			&\sum\limits_{i=1}^{n}(a_i^{\frac{1}{vn}})^{un} \geq \sum\limits_{i=1}^{n} \left( \prod\limits_{j=i}^{\sigma (i)-1 +\Delta_n(i)}(a_j^{\frac{1}{vn}})^{vn} \cdot \left( \prod\limits_{j=1}^n a_j^{\frac{1}{vn}}\right)^{v(\frac{u}{v}+i-\sigma(i) -\Delta_n(i))}\right).
		\end{align*}
		If $\sigma$ were a cyclic shift, we would immediately recognize the inequality in Theorem \ref{rear} for $un$ copies of the sequence $(a_1^\frac{1}{vn}, \ldots, a_n^\frac{1}{vn})$, and both Corollaries \ref{wieszift} and \ref{mnieszift} are proven. In the general case however, we need some more work before we can use Theorem \ref{rear}.%AC09.21 changed from "we need to put some work to use Theorem 2.1."
		
		Let $\alpha$ and $\beta$ be permutations of $\{1,\ldots, m\}$ for some natural $m$. If the sum of $m$ fractions $\sum\limits_{i=1}^{m}\frac{x_{\alpha(i)}}{x_{\beta(i)}}$, with $x_i$'s pairwise different, can be sorted so that the denominator of the $i$-th fraction is equal to the numerator of $(i+1)$-th fraction we say the sum is \emph{cyclically constructed}. This of course happens iff $\alpha(\gamma(i+1))=\beta(\gamma(i))$ for some permutation $\gamma$ (the sorting).
		\begin{lemma}\label{samenumber}
			The cyclically constructed (and accordingly sorted) sum of $m$ fractions $\sum\limits_{i=1}^{m}\frac{x_{\alpha(i)}}{x_{\beta(i)}}$ rewritten as the sum of $m$ products of consecutive $a_i=\frac{x_i}{x_{i+1}}$ in the following manner (note the abuse of notation in (re)definition of $\Delta$):
			\begin{align*}
				\sum\limits_{i=1}^m \frac{x_{\alpha(i)}}{x_{\beta(i)}}=\sum_{i=1}^m \prod\limits_{j=\alpha(i)}^{\beta(i)-1+\Delta_m(i)}a_j, && \text{ where } && \Delta_m(i)=\begin{cases}0 & \text{ for } \alpha(i) \leq \beta(i)\\
					m & \text{ for } \alpha(i) \geq \beta(i)
				\end{cases}
			\end{align*}
			contains the same number of each $a_1,\ldots,a_m$.
		\end{lemma}
		\begin{proof}
			Rewriting the sum on the RHS in its expanded form, it starts with $a_{\alpha(1)}$ and afterwards passes through the consecutive $a_j$'s up to $a_{\beta(1)-1+\Delta_m(1)}$. The next summand picks up at $a_{\alpha(2)}=a_{\beta(1)+\Delta_m(1)}$, by the cyclically constructed condition, and so forth. This condition also implies that the last factor of the last summand is $a_{\alpha(1)-1}$, thus ensuring that we passed every $a_j$ an equal number of times.
		\end{proof}
		\begin{lemma}
			The sum $\sum\limits_{i=1}^{n} \frac{x_i}{x_{\sigma (i)}}$ can be rewritten as a sum of cyclically constructed sums of fractions and some integer.
		\end{lemma}
		\begin{proof}
			We will describe an algorithm to produce such a representation, using the cycle decomposition of $\sigma$.
			\begin{enumerate}
				\item For any $i\in\{1,\ldots,n\}$ such that $\sigma (i)=i$, we have $\frac{x_i}{x_{\sigma (i)}}=1$, and thus the sum over all such $i$'s is an integer. Exclude such $i$'s. %AC09.21 reworked the sentence to emphasize which i's are picked;
				\item The first $i$ not yet excluded lies in a cycle $\{i,\sigma(i),\ldots,\sigma^m(i)\}$, for some $m$. Then $\sum\limits_{j=1}^{m}\frac{x_{\sigma^j(i)}}{x_{\sigma^{j+1}(i)}}$ is cyclically constructed. Exclude this orbit of~$i$.
				\item Repeat step (2) until all $i$'s are excluded.
			\end{enumerate}
			This procedure terminates and produces a representation as claimed.
		\end{proof}
		We are now ready to get back to the inequality 
		\begin{align*}
			\sum\limits_{i=1}^{n}(a_i^{\frac{1}{vn}})^{un} \geq \sum\limits_{i=1}^{n} \left( \prod\limits_{j=i}^{\sigma (i)-1 +\Delta_n(i)}(a_j^{\frac{1}{vn}})^{vn} \cdot \left( \prod\limits_{j=1}^n a_j^{\frac{1}{vn}}\right)^{v(\frac{u}{v}+i-\sigma(i) -\Delta_n(i))}\right).
		\end{align*}
		We claim it holds by Theorem \ref{rear} applied to $un$ copies of the sequence $(a_1^\frac{1}{vn}, \ldots, a_n^\frac{1}{vn})$: note that its LHS is invariant when this sequence is sorted to be increasing, and we again recognize the LHS of the Rearrangement Inequality. We claim that the RHS is too in an appropriate form for Theorem \ref{rear} to be applicable. Note that the RHS fulfills conditions of Lemma \ref{samenumber}, so it contains the same number of each $a_i^\frac{1}{vn}$.
		
		It will be sufficient to find an appropriate section by the Hall's marriage theorem. Indeed, in a family of sets $\{S^1_i\}_{i=1}^{n}$, where $S^1_i$ consists of terms $a_j$ present in the product in the $i$-th summand on the RHS of the last inequality, the marriage condition is satisfied, because each $a_j$ occurs in the sum $un$ times and each summand is a product of exactly $un$ terms. Thus, a section $(a_1^\frac{1}{vn}, \ldots, a_n^\frac{1}{vn})$ exists. We can now form the family $\{S^2_i\}_{i=1}^n$ of sets of terms $a_j$ appearing in the product in the $i$-th summand divided by the element previously selected, which again satisfies the marriage condition. Iterating this process, we manage to present the RHS as claimed, thus proving the theorem under the first condition for rational $k$ - and of course the non-rational case follows by continuity. The proof for negative $k$ follows the same steps, but starting with $k=-\frac{u}{v}$ and working with sequence $b_i$ instead of $a_i$.%AC09.21 changed the second sentence: "product in n-th summand" to "$a_j$ present in the product in the $i$-th summand" (also fixed the $i$-th summand, instead of prior $n$-th, which was incorrect); changed the last sentence from "The proof for negative k falls along exactly the same steps";
	\end{proof}
	\begin{remark}
		We note that for $k>1$, inequality \eqref{original} is not an inequality in $p$-norms (suppose $P_{\sigma}(x_1,\ldots,x_n):=\left(\frac{x_1}{x_{\sigma(1)}},\ldots,\frac{x_n}{x_{\sigma(n)}}\right)$, then the RHS of \eqref{original} is $\left\Vert P_{\sigma}(x_1,\ldots,x_n)\right\Vert_{1}$ while the LHS is $\left\Vert P_{i\mapsto i+1}(x_1,\ldots,x_n)\right\Vert_{k}^{k}$). We cannot take the $k$-th root of the LHS and preserve the inequality, since for $x_1=\ldots=x_n$ both sides are equal (and greater than 1). For the sake of completeness, note that for $0<k<1$ we do have an inequality between the $k$-F-norm on the left and the 1-norm on the right, although we can only have the identity permutation there.
	\end{remark}%AC09.21 changed the first sentence from "does not spell out"
	\begin{remark}
		The permutations of $\{1,\ldots,n\}$ restricted by $k$ appearing in Theorem \ref{result} can be viewed as a variation of the combinatorial menag\'e problem. Their number $P_{n,k}$ can be in principle computed by the rook polynomials or the methods from Chapter 4.7 of \cite{wieze}. Apart from the trivial cases ($P_{n,0}=1$, $P_{n,1}=2$), we note here that this book gives the explicit computation of $P_{n,2}$ in Example 4.7.9 (after a permutation of the chessboard rows) and observes that they follow the Lucas numbers, $P_{n,2}=2+L_n$ save for the two initial outliers.
	\end{remark}
	\section*{Counterexamples}
	Suppose that for $k\geqslant 0$, there exist $i_0 \in \{1, 2, .. , n\} $ for which $ \Delta_n(i_0)+\sigma (i_0)-i_0 >k$.
	
	Fix $R>0$, such that $\Delta_n(i_0)+\sigma (i_0)-i_0 =k+R$. If we find an $n$-tuple $(x_i)$ such that $a_i=a_j>n^{\frac{1}{R}}\geqslant 1$, for $i, j\in \{1, 2, ..., n\} \setminus\{i_0-1\}$ (note that we then have $ a_{i_0-1}=\frac{1}{a_{i_0}^{n-1}}<a_{i_0}$), the following inequalities hold: 
	\begin{align*}
		\sum_{i=1}^{n}a_i^k=(n-1)a_{i_0}^k&+a_{i_0-1}^k<na_{i_0}^k=na_{i_0}^{\sigma (i_0)-i_0-R +\Delta_n(i_0)}\\
		\frac{n}{a_{i_0}^R}\cdot a_{i_0}^{\sigma (i_0)-i_0 +\Delta_n(i_0)}&<a_{i_0}^{\sigma (i_0)-i_0 +\Delta_n(i_0)}\\
		&=\prod_{j=i_0}^{\sigma (i_0)-i+\Delta_n(i_0)}a_j<\sum_{i=1}^{n} \prod_{j=i}^{\sigma (i)-i+\Delta_n(i_0)}a_j.
	\end{align*}
	An example of such $n$-tuple is $x_{i_0-1}:=2$ (or any arbitrary positive constant), $x_{i_0-1-i}:=x_{i_0-1}\left((n+1)^{\frac{1}{R}}\right)^{i}$. Similarly, the $n$-tuple $x_{i_0+1}:=2$ (or again, any arbitrary positive constant), $x_{i_0+1+i}:=x_{i_0+1}\left((n+1)^{\frac{1}{R}}\right)^{i}$, gives a counterexample for any negative $k$ outside of the range of Theorem \ref{result}.
	
	\section{Shapiro-type inequalities with exponent weights and permutations}\label{szap}
	
	In this section we briefly touch upon Shapiro (or Nesbitt's) inequality. We will distinguish the ``Shapiro-type inequalities with exponent weights'' of the form
	\begin{align*}\tag{\ref{szapiro}}
		\sum\limits_{i=1}^{n}\left(\frac{x_i}{x_{i+1}+x_{i+2}}\right)^k\geq\sum\limits_{i=1}^{n}\frac{x_i}{x_{\sigma(i)}+x_{\sigma^2(i)}}
	\end{align*}
	and ``Shapiro inequalities with exponent weights'', when the permutation makes the RHS constant and equal to $\frac{n}{2}$. This is of course the case iff $\sigma$ is a product of disjoint transpositions.
	
	Theorem \ref{result} may be taken to suggest that sufficiently similar inequalities should exhibit a similar behaviour: the inequality \eqref{main} can be read to mean that the dampening by exponents on the LHS can accommodate a sufficiently ``small'' action of a permutation on the RHS. It is not unreasonable to expect that this effect would carry over in some form to Shapiro-type inequalities (e.g., because the rational functions involved are of the same degree). We will show when Shapiro-type inequalities hold, and observe that there is no interdependence between $k$ and $\sigma$. To that end, in this section we assume that $k$ is non-negative. We again tacitly assume everywhere that $x_1,\ldots, x_n$ are positive. %AC06.21 added the assumption per the referee's request; %AC09.21 added missing articles, and added the quotation marks to emphasize informal use of the word "small"; changed "how" to "that" before "there is no interdependence";
	\begin{example}\label{jeden}
		For $n=2$, the RHS of \eqref{szapiro} is equal to 1 for any of the two permutations of two elements. On the LHS we see a function of a parameter in the interval $[0,1]$, $t^k+(1-t)^k$ which attains a minimum less than 1 for $k>1$ and is always greater than or equal to 1 for every $k\leq 1$.
	\end{example}%AC09.21 added "than" in "greater or equal"
	\begin{example}\label{nes}
		For $n=3$, on the RHS of \eqref{szapiro} we have the following two cases:
		\begin{enumerate}
			\item if $\sigma$ is a 3-cycle, then
			\begin{itemize}
				\item for $k=1$ we have an equality;
				\item for $k>1$ $x_1=x_2=x_3=1$ is a counterexample;
				\item for any $k<1$, a sufficiently large $x_1>>1$, and $x_2=x_3=1$ we have the LHS of \eqref{szapiro} lesser than $\left(\frac{x_1}{2}\right)^k + 1 +1$, and the RHS greater than $\frac{x_1}{2}$ -- and since the latter tends to infinity faster than the former, we have a counterexample;%AC06.21 reworked the counterexample per the referee request;
			\end{itemize}
			\item if $\sigma$ is either a 2-cycle or the identity, then the RHS of \eqref{szapiro} is equal to $\frac{3}{2}$, and
			\begin{itemize}
				\item for $k=1$ we have the Shapiro inequality, which does hold in dimension 3;
				\item for $k>1$, $x_1=x_2=x_3=1$ is a counterexample;
				\item for $k<1$ the inequality holds by Proposition \ref{jens} below.
			\end{itemize}
		\end{enumerate}
	\end{example}
	\begin{example}\label{trzy}
		Finally, the two options for $n\geq 4$ are:
		\begin{enumerate}
			\item if the RHS of \eqref{szapiro} is not uniformly $\frac{n}{2}$ (i.e. assuming that $\sigma$ is not a product of disjoint transpositions), then
			\begin{itemize}
				\item for $k=1$ and $\sigma(i)=i+1$, we have an equality;
				\item for $k=1$ and $\sigma$ containing any other cycle, we will have an $i$ such that $\sigma(i)$ and $\sigma^2(i)$ are not two consecutive numbers mod $n$ (and all three are pairwise different), thus getting a counterexample with $x_{\sigma(i)}=x_{\sigma^2(i)}=r<<1$ and with all the other $x_j=1$: all terms on the LHS are then bounded by $1$ while there is the unbounded term $\frac{1}{2r}$ on the RHS;%AC09.21 changed "such an i that" to "an i such that" 
				\item for $k>1$ we have the counterexample $x_1=\ldots =x_n=1$;
				\item for $k<1$ taking such $i$ that $\sigma(i)\neq\sigma^2(i)$ and putting $x_i>>1$ with all other $x_j=1$ we again have a counterexample as in bullet 3 of point 1 of Example \ref{nes}.
			\end{itemize}
			\item if the RHS is equal to $\frac{n}{2}$, then
			\begin{itemize}
				\item for $k>1$ we have the counterexample $x_1=\ldots =x_n=1$;
				\item for $k\leq 1$ is subject to Proposition \ref{jens} below.
			\end{itemize}
		\end{enumerate}
	\end{example}
	Having treated all Shapiro-type inequalities we move to Shapiro inequalities with exponents, when the RHS is uniformly equal to $\frac{n}{2}$. 
	\begin{proposition}\label{jens}
		The Shapiro inequality with exponents in dimension $n$
		\begin{align}\label{szsz}
			\sum\limits_{i=1}^{n}\left(\frac{x_i}{x_{i+1}+x_{i+2}}\right)^k\geq\frac{n}{2},
		\end{align}
		where $x_1,\ldots, x_n$ are positive real numbers, holds for every $k\leq 1$ iff the original Shapiro inequality does hold in this dimension.
	\end{proposition}
	We note the work of Daykin \cite{day} and Diananda \cite{dia} on related problems. We want to thank Maciej Raczuk and independently Piotr Kumor for pointing to us the reasoning below, greatly simplifying our original proof.%AC09.21 reworked the sentence from "making the convoluted original analytical proof obsolete"
	\begin{proof}
		For $n=3$ this follows from Theorem \ref{rear} and convexity argument, but we tackle the general case. Since for $k_1\leq k_2 \leq 1$ we have
		\begin{align*}
			\left(\frac{x_i}{x_{i+1}+x_{i+2}}\right)^{k_1}=\left(\left(\frac{x_i}{x_{i+1}+x_{i+2}}\right)^{\frac{k_1}{k_2}}\right)^{k_2}\geq \left(\frac{x_i^{\frac{k_1}{k_2}}}{x_{i+1}^{\frac{k_1}{k_2}}+x_{i+2}^{\frac{k_1}{k_2}}}\right)^{k_2}
		\end{align*}
		because $x\mapsto x^{\frac{k_1}{k_2}}$ is concave, and therefore it follows that the the infimum of the LHS of \eqref{szsz} for $k_1$ is bounded from below by the infimum for $k_2$. Therefore they are all bounded by the infimum for $k=1$ and the proof is completed: in the dimensions where the original Shapiro inequality holds, this infimum is equal $\frac{n}{2}$; and in the dimensions in which it fails, any counterexample to the original Shapiro inequality is also a counterexample to the Shapiro inequality with every exponent $k$ sufficiently close to 1 (by continuity). Note that how close precisely $k$ must be to 1 to give a counterexample seems a hard problem. %AC09.21 reworked the sentence hat contained "since in every dimension that a counterexample for the original Shapiro inequality exists" 
	\end{proof}
	\begin{remark}[Nesbitt's inequality with exponents]
		Shapiro inequality in dimension 3, as in Example \ref{nes} is often called Nesbitt's inequality. The common generalization is
		\begin{align}\label{nesb}
			\frac{x_1}{x_2+x_3+\ldots + x_n}+\ldots+\frac{x_n}{x_1+x_2+\ldots + x_{n-1}}\geq \frac{n}{n-1},
		\end{align}
		which is often proved by Theorem \ref{rear}: assuming $x_1\geq\ldots\geq x_n$ denote $s_i=\sum\limits_{j\neq i}x_j$ and note that $\frac{1}{s_1}\geq\ldots\geq\frac{1}{s_1}$ and applying the Rearrangement inequality $n-1$ times we have
		\begin{align*}
			(n-1)\sum\limits_{i=1}^{n}\frac{x_i}{s_i}\geq \sum\limits_{i=1}^{n}\frac{s_i}{s_i}\geq n
		\end{align*}
		the last inequality written as such to immediately observe that this proof generalizes by the concavity of $x\mapsto x^k$ to
		\begin{align}\label{nesbb}
			\sum\limits_{i=1}^n\left(\frac{x_i}{s_i}\right)^k\geq \frac{n}{n-1},
		\end{align}
		for every $k\leq 1$.
		
		A version of that inequality but ``with exponent weights'' was proven in \cite{nesb} (but note also \cite{wang}), namely Theorem 3 there states that for positive real numbers $x_1,\ldots,x_n$ and $k\geq 1$ the following inequality holds
		\begin{align}\label{nesbb}
			\sum\limits_{i=1}^n\left(\frac{x_i}{s_i}\right)^k\geq \frac{n}{(n-1)^k}.
		\end{align}
		One might want to extend this inequality to $k\leq 1$ at least for $n=3$ using Proposition \ref{jens} (e.g. since the minima of LHS tend to be attained in the corner $x_1=\ldots=x_n=1$ where \eqref{nesbb} does hold). However, a counterexample exists and for $x=1$, $y=0.1$, $z=0.1$, $k=0.1$ we have the opposite inequality
		\begin{align*}
			\left(\frac{x}{y+z}\right)^k+\left(\frac{y}{x+z}\right)^k+\left(\frac{z}{x+y}\right)^k< \frac{3}{2^k}.
		\end{align*}
	\end{remark}
	
	%%%%%%%% Bibliography %%%%%%%%%%%%%%%%%%%%%%%%%%%%%%%%%%%%%%%%%%%%%%%%%

\end{document}